\documentclass[11pt]{article}

\usepackage{CJK}
\usepackage{amsmath}
\usepackage{bm}
\usepackage{amsfonts}
\usepackage{latexsym}
\usepackage{amssymb}
\usepackage{stmaryrd}
\usepackage{overpic}
\usepackage{graphicx}
\usepackage{multirow}
\usepackage{mathrsfs}
\usepackage{float}
\usepackage{color}
\usepackage{setspace}
\usepackage{titling}
\usepackage{subfigure}
\usepackage{dsfont}
\usepackage{appendix}

\newtheorem{theorem}{Theorem}[section]

\newtheorem{corollary}[theorem]{Corollary}
\newtheorem{definition}[theorem]{Definition}

\newtheorem{lemma}[theorem]{Lemma}

\newenvironment{proof}[1][Proof]{\textbf{#1.} }
{\ \rule{0.75em}{0.75em}\smallskip}

\numberwithin{equation}{section}

\title{Finite Element Calculation of Photonic Band Structures for Frequency Dependent Materials}
\author{Wenqiang Xiao \thanks{Beijing Computational Science Research Center, Beijing 100193, China. ({\tt wqxiao@csrc.ac.cn})} \and
Bo Gong \thanks{Beijing Computational Science Research Center, Beijing 100193, China. ({\tt gongbo@csrc.ac.cn})} \and
Jiguang Sun \thanks{Department of
Mathematical Sciences, Michigan Technological University, Houghton, MI 49931, U.S.A. ({\tt  jiguangs@mtu.edu}).}
\and Zhimin Zhang \thanks{Beijing Computational Science Research Center, Beijing 100193, China and Department of Mathematics, Wayne State University, Detroit, MI 48202, USA. ({\tt zmzhang@csrc.ac.cn}).}
}

\begin{document}
\maketitle

\begin{abstract}
We consider the calculation of the band structure of frequency dependent photonic crystals.
The associated eigenvalue problem is nonlinear and
it is challenging to develop effective convergent numerical methods.
In this paper, the band structure problem is formulated as the eigenvalue problem of a holomorphic Fredholm operator function of index zero. 
Lagrange finite elements are used to discretize the operator function.  Then 
the convergence of the eigenvalues is proved using the abstract approximation theory for holomorphic operator functions.
A spectral indicator method is developed to practically compute the eigenvalues. 
Numerical examples are presented to validate the theory and show the effectiveness of the proposed method.
\end{abstract}

\section{Introduction}
Photonic crystals (PCs) are periodic structures composed of dielectric materials and exhibiting band gaps. 
Band gaps contain prohibited frequencies of electromagnetic waves going through such crystals. 
This phenomenon has many important applications in optical communications, filters, lasers,  switches and optical transistors, etc. 
\cite{kuzmiak1997, axmann1999, pendry1996, spence2005, raman2010, sakoda2001-II}. Effective calculations of the band
structures are essential to identify novel optical phenomena and develop new devices.

For 2D PCs, the Maxwell's equations are reduced to scalar wave equations. 
The Floquet-Bloch theory \cite{kuchment1993} transforms the band structure problem on the infinite periodic structure to 
an eigenvalue problem on the unit cell (see Fig.~\ref{domain and k}: left). 
If the permittivity is independent of the frequency, the eigenvalue problem is linear. 
There exist many successful numerical methods such as 
the plane wave expansion method \cite{ho1990, figotin1997, johnson2001, norton2010}, the finite-difference time-domain (FDTD) method \cite{qiu2000}, 
and the finite element method \cite{axmann1999, dobson1999, boffi2006}. 
In contrast, frequency dependent permittivities usually lead to nonlinear eigenvalue problems. 
Various methods have been developed as well, e.g., 
generalizations of the plane wave expansion method \cite{kuzmiak1994, kuzmiak1998, gu2006}, FDTD method \cite{ito2001}, 
the cutting plane method \cite{toader2004}, transfer matrix method \cite{pendry1996},
and path-following algorithm \cite{spence2005}.
Most methods for solving nonlinear eigenvalue problems are based on the Newton's iteration \cite{ruhe1973} 
or extensions of the techniques for linear problems \cite{voss2004, sleijpen1996}. 
However, Newton type methods often need good initial estimates of the eigenvalues/eigenvectors
and linearization techniques require restrictions on the nonlinear permittivities \cite{mackey2006}. 
Moreover, the complexity of the spectrum due to the nonlinearity makes the convergence study much more difficult.

In this paper, considering the permittivity as a general nonlinear function of the frequency, we propose a new finite element approach for the band structure calculation 
of frequency dependent materials.
Firstly, the band structure is formulated as the eigenvalue problem of a holomorphic Fredholm operator function. 
Secondly, Lagrange finite elements are used for discretization and the convergence is proved using the abstract approximation theory for the eigenvalue problem
of the holomorphic operator function \cite{karma1996I, karma1996II}.
Thirdly, a spectral indicator method (SIM) is developed to practically compute the eigenvalues.
This version of SIM is similar to that in \cite{gong2020} for the nonlinear transmission eigenvalues (see also \cite{xiao2020} for Dirichlet eigenvalues). 
It extends the idea in \cite{huang2016, huang2018} for the generalized eigenvalues of non-Hermitian matrices and is particularly 
effective to compute eigenvalues of a holomorphic operator function. 

The rest of the paper is organized as follows. In Section 2, some preliminaries of holomorphic Fredholm operator functions and the associated abstract approximation theory are presented. 
In Section 3, we introduce the mathematical model of the band structures for 2D photonic crystals and formulate the problem as the eigenvalue problem of a holomorphic Fredholm operator function of index zero.
In Section 4, Lagrange finite elements are used for discretization and the convergence is proved using the abstract approximation theory of Karma \cite{karma1996I, karma1996II}. 
Section 5 contains a spectral indicator method to compute 
the eigenvalues in a region on the complex plane. Numerical examples are shown in Section 6. Finally, we draw some conclusions and discuss future work in Section 7.

\section{Preliminaries}
We present some preliminaries on holomorphic Fredholm operator functions and the abstract approximation 
theory for the associated eigenvalue problems (see, e.g., \cite{gohberg2009, karma1996I, karma1996II, beyn2014}).
Let $X,Y$ be complex Banach spaces and $\Omega \subset\mathbb{C}$ be compact and simply connected.
Denote by $\mathcal{L}(X,Y)$ the space of bounded linear operators from $X$ to $Y$.
\begin{definition}
An operator $A \in \mathcal{L}(X, Y)$ is said to be Fredholm if
\begin{itemize}
\item[1.] the range of $A$, denoted by $\mathcal{R}(A)$, is closed in $Y$;
\item[2.] the null space of $A$, denoted by $\mathcal{N}(A)$, and the quotient space $Y/\mathcal{R}(A)$ are finite-dimensional.
\end{itemize}
The index of $A$ is the integer defined by
\[
\text{ind}(A) = \dim \mathcal{N}(A) - \dim (Y/\mathcal{R}(A)).
\]
\end{definition}
\begin{definition}
Let $E$ be a Banach space and $\Omega \subset \mathbb C$ be an open set. A function $f: \Omega \to E$ is called holomorphic if, for each $w \in \Omega$,
\[
f'(w) := \lim_{z \to w} \frac{f(z)-f(w)}{z-w}
\]
exists.
\end{definition}

Let $T: \Omega \to \mathcal{L}(X, Y)$ be a holomorphic operator function on $\Omega$.
Denote by $\Phi_0(\Omega,\mathcal{L}(X,Y))$ the set of holomorphic Fredholm operator functions of index zero \cite{gohberg2009}.
Assume that $T \in \Phi_0(\Omega,\mathcal{L}(X,Y))$, i.e., for each $\omega \in \Omega$, $T(\omega) \in  \mathcal{L}(X, Y)$ is a Fredholm operator of index zero. 
The eigenvalue problem is to find $(\omega,u)\in \Omega \times X,~u\neq 0$, such that
\begin{equation}\label{nonlinear eig}
T(\omega)u=0.
\end{equation}
The resolvent set $\rho(T)$ and the spectrum $\sigma(T)$ of $T$ with respect to $\Omega$ are, respectively, defined as
\[
\rho(T)=\{\omega\in \Omega: T(\omega)^{-1} \in \mathcal{L}(Y,X) \}\quad \text{and} \quad \sigma(T)=\Omega \backslash\rho(T).
\]
Throughout the paper, we assume that $\rho(T)\neq\emptyset$. Then the spectrum $\sigma(T)$ has no cluster points in $\Omega$
and every $\omega \in \sigma(T)$ is an eigenvalue \cite{karma1996I}.

To approximate the eigenvalues of $T$, we consider operator functions 
$T_n\in \Phi_0(\Omega,\mathcal{L}(X_n,Y_n))$, $n\in\mathbb{N}$, such that the following properties hold \cite{karma1996I, beyn2014}.
\begin{itemize}
\item[(A1)] There exist Banach spaces $X_n,Y_n$, $n\in\mathbb{N}$, and linear bounded mappings $p_n\in\mathcal{L}(X,X_n)$, $q_n\in\mathcal{L}(Y,Y_n)$ such that
\begin{equation*}\label{pnqn}
\lim\limits_{n \rightarrow\infty}\|p_n v\|_{X_n}=\|v\|_X,\, v\in X,\quad
\lim\limits_{n \rightarrow\infty}\|q_n v\|_{Y_n}=\|v\|_Y,\, v\in Y.
\end{equation*}
\item[(A2)] The sequence $\{T_n(\cdot)\}_{n \in \mathbb{N}}$ satisfies
\[
\|T_n(\omega)\| \le \infty \quad \text{for all } \omega \in \Omega, n \in \mathbb{N}.
\]
\item[(A3)] $\{T_n(\omega)\}_{n \in \mathbb{N}}$ approximates $T(\omega)$ for every $\omega\in \Omega$, i.e., 
\[
\lim_{n \to \infty} \|T_n(\omega)p_n x - q_n T(\omega) x\|_{Y_n} = 0 \quad \text{for all } x \in X. 
\]
\item[(A4)] Let $\omega \in \Omega$ and $\{x_n\} \subset X_n, n\in  \mathbb{N}$ be 
bounded. For any subsequence  $\{x_n\}$, $n\in N \subset \mathbb N$,
\[
\lim_{\mathbb{N} \ni n\rightarrow \infty} \|T_n(\omega) x_n - q_n y\|_{Y_n}=0
\]
for some $y\in Y$, there exists a subsequence $N' \subset N$
and an element $x\in X$ such that 
\[
\lim\limits_{N' \ni n\rightarrow \infty} \|x_n-p_n x\|_{X_n}=0.
\]
\end{itemize}

If the above conditions are satisfied, one has the following abstract approximation result (see Section 2 of \cite{karma1996II} or Theorem 2.10 of \cite{beyn2014}).

\begin{theorem}\label{theorem1}
Assume that {\rm (A1)-(A4)} hold. For any $\omega\in\sigma(T)$ there exists $n_0\in \mathbb{N}$
and a sequence $\omega_n\in \sigma(T_n), n\ge n_0$, such that $\omega_n\to \omega$ as 
$n\to \infty$. Furthermore, letting $G(\omega)$ be the generalized eigenspace of $\omega$ and
$r_0$ be the maximum rank of eigenvectors associated to $\omega$, it holds that
\[
|\omega_n-\omega|\leq C \varepsilon_n ^{1/{r_0}},
\]
where 
\[
\varepsilon_n=\max\limits_{|\eta-\omega|\leq \delta} \max\limits_{v\in G(\omega)}\|T_n(\eta) p_n v-q_n T(\eta)v\|_{Y_n},
\]
for sufficiently small $\delta> 0$. 
\end{theorem}

\section{The Mathematical Model}
Photonic band structures are posted as the eigenvalue problems of the Maxwell's equations \cite{axmann1999, raman2010}
\begin{equation}\label{maxwell}
\left\{
\begin{aligned}
& \nabla\times\frac{1}{\epsilon(x,\omega)}\nabla\times\mathbf{H}=\left(\frac{\omega}{c}\right)^2\mathbf{H},\\
& \nabla\cdot\mathbf{H}=0,
\end{aligned}
\right.
\end{equation}
where $\mathbf{H}$ is the magnetic field, $\omega$ is the frequency, $\epsilon(x,\omega)$ is the electric permittivity, and $c$ is the speed of light in the vacuum.
In 2D, the Maxwell's equations \eqref{maxwell} can be reduced to the transverse electric (TE) case or the transverse magnetic (TM) case.
In this paper, we consider the TE case 
\begin{equation}\label{TE}
-\Delta\psi=\left(\frac{\omega}{c}\right)^2\epsilon(x,\omega)\psi.
\end{equation}

Assume that the photonic crystal has the unit periodicity on a square lattice. Letting
$\mathbb{Z} =\{0,\pm 1,\pm 2,\cdots\}$ and defining the lattice $\Lambda=\mathbb{Z}^2$, one has that
\begin{equation*}
\epsilon(x+n,\omega)=\epsilon(x,\omega)~~~{\rm for~all} ~x\in\mathbb{R}^2 \text{ and } n\in\Lambda.
\end{equation*}
The periodic domain $D:=\mathbb{R}^2/\mathbb{Z}^2$
can be identified with the unit square $D_0:=(0, 1)^2$ by imposing periodic boundary conditions (Fig.~\ref{domain and k}: left).
Introduce the so-called quasimomentum vector ${\boldsymbol k}\in \mathcal{K}$ (Fig.~\ref{domain and k}: right), where
$$\mathcal{K}=\{{\boldsymbol k}\in \mathbb{R}^2|-\pi\leq k_j\leq\pi,j=1,2\}$$
is called the Brillouin zone.
Using the Floquet transform \cite{kuchment1993}, \eqref{TE} 
can be written as an eigenvalue problem parametrized by ${\boldsymbol k}$ on the primitive cell $D_0$ \cite{axmann1999}
\begin{equation}\label{TE-equation}
 -(\nabla+i {\boldsymbol k})\cdot(\nabla+i{\boldsymbol k})u(x)=\left(\frac{\omega}{c}\right)^2\epsilon(x,\omega)u(x) \quad {\rm in}~D_0,
 \end{equation}
where $u$ is the Floquet transform of $\psi$, a periodic function in both $x_1$ and $x_2$.
 
We assume that, for almost all $x\in \mathbb{R}^2$, $\epsilon(x,\omega)$ is holomorphic on $\Omega$ and satisfies
\begin{equation}\label{boundness}
0<C_0\le |\epsilon(x,\omega)|\le C_1<\infty,~x\in \mathbb{R}^2, \omega\in \Omega.
\end{equation}
 
Let $\Gamma_i , i=1,\cdots,4,$ be the four parts of $\partial D_0$ (Fig.~\ref{domain and k}: left).  The eigenvalue problem in the TE case is to find $\omega$ and $u$ such that, for a fixed ${\boldsymbol k}\in\mathcal{K}$,
\begin{equation}\label{TE pde}
\left\{
\begin{aligned}
&-(\nabla+i{\boldsymbol k})\cdot(\nabla+i{\boldsymbol k})u(x)=\left(\frac{\omega}{c}\right)^2\epsilon(x,\omega)u(x),~x\in D_0, \\
&u|_{\Gamma_1}=u|_{\Gamma_2},~~u_{x_1}|_{\Gamma_1}=u_{x_1}|_{\Gamma_2},\\
&u|_{\Gamma_3}=u|_{\Gamma_4},~~u_{x_2}|_{\Gamma_3}=u_{x_2}|_{\Gamma_4}.
\end{aligned}
\right.
\end{equation}

\begin{figure}[h!]
\begin{center}
\begin{tabular}{lll}
\resizebox{0.48\textwidth}{!}{\includegraphics{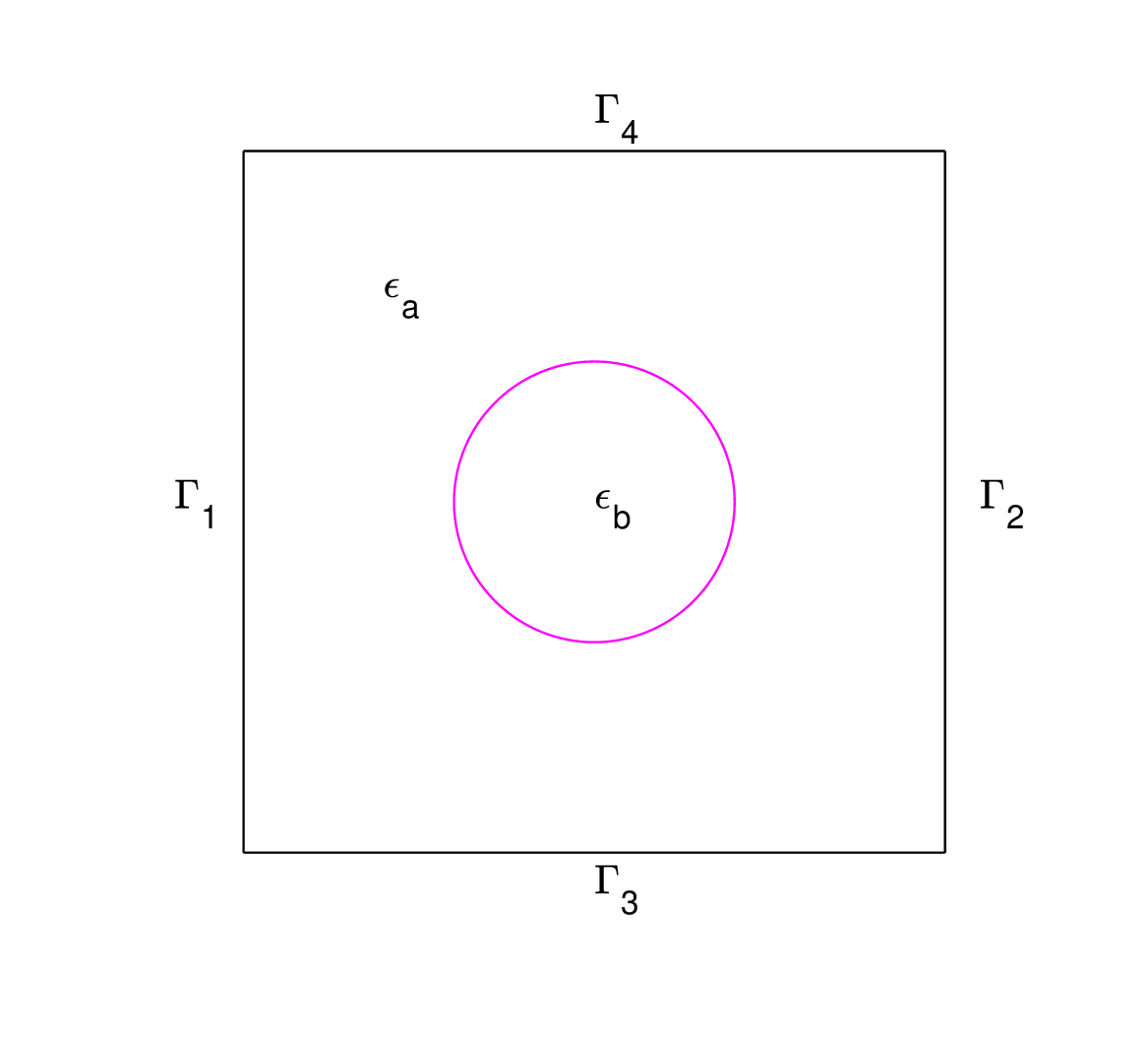}}&
\resizebox{0.48\textwidth}{!}{\includegraphics{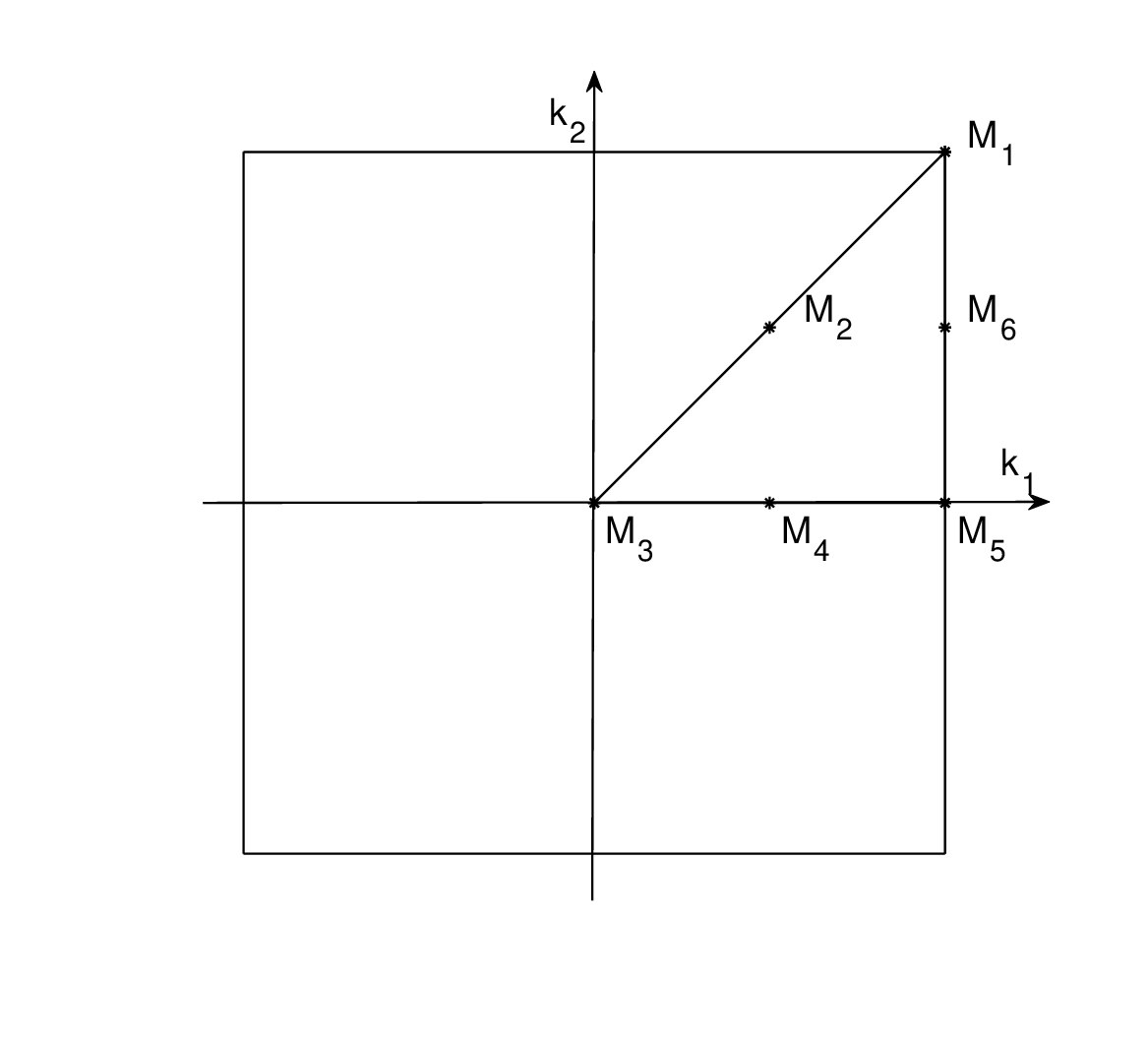}}
\end{tabular}
\end{center}
\caption{Left: a square cell $D_0$ with a disc of radius $r$ at its center.  Right: Brillouin zone $\mathcal{K}$: $M_1=(\pi,\pi), M_2=(\frac{\pi}{2},\frac{\pi}{2}), M_3=(0,0), M_4=(\frac{\pi}{2},0), M_5=(\pi,0), M_6=(\pi,\frac{\pi}{2})$.}
\label{domain and k}
\end{figure} 


We first formulate \eqref{TE pde} 
as the eigenvalue problem of a holomorphic Fredholm operator function of index zero (see \cite{engstrom2010}).
Let $H^1(D_0)$ be the Sobolev space of functions in $L^2(D_0)$ with square integrable gradients.
Denote by $(\cdot,\cdot)_1$ and $\|\cdot\|_1$ the inner product and norm on $H^1(D_0)$, respectively. 
Define the subspace of functions in $H^1(D_0)$ with periodic boundary conditions by
\begin{equation}
H_p^1(D_0)=\{v\in H^1(D_0)|~v|_{\Gamma_1}=v|_{\Gamma_2},v|_{\Gamma_3}=v|_{\Gamma_4}\}.
\end{equation}
Multiplying \eqref {TE-equation} by a test function $v\in H_p^1(D_0)$ and integrating by parts, the weak formulation is to find 
$(\omega, u) \in \Omega \times H_p^1(D_0)$ such that
\begin{equation}\label{TE weak form}
\int_{D_0}(\nabla+i{\boldsymbol k})u\cdot\overline{(\nabla+i{\boldsymbol k})v}dx=\left(\frac{\omega}{c}\right)^2\int_{D_0} \epsilon(x,\omega)u\overline{v}dx \quad \text{for all } v\in H_p^1(D_0).
\end{equation}
If $\epsilon(x, \omega)$ depends on $\omega$,  \eqref{TE weak form} is a nonlinear eigenvalue problem in general.

Rewrite the sesquilinear form \eqref{TE weak form} as (see \cite{engstrom2010})
\begin{equation}\label{eqn1}
(u,v)_1-b(\omega; u,v)=0,
\end{equation}
where
\begin{equation}\label{bTE}
b(\omega; u,v)=\int_{D_0} \left(\left(\frac{\omega}{c}\right)^2\epsilon+1\right)u\overline{v}-2iu{\boldsymbol k}\cdot\nabla\overline{v}-|{\boldsymbol k}|^2u\overline{v}dx, ~\omega\in \Omega.
\end{equation}

Due to boundedness of $\epsilon$ stated in \eqref{boundness}, $b(\omega; \cdot, \cdot)$ is bounded. 
By the Riesz representation theorem, we define an operator function $B: \Omega \rightarrow \mathcal{L}(H_p^1(D_0),H_p^1(D_0))$ such that
\begin{equation}\label{BTE}
b(\omega; u,v)=(B(\omega)u,v)_1 \quad \text{for all } v\in H_p^1(D_0).
\end{equation}
Using \eqref{eqn1}, \eqref{TE weak form} can be written as
\begin{equation}\label{operator form}
(u,v)_1-(B(\omega)u,v)_1=0 \quad \text{for all } v\in H_p^1(D_0).
\end{equation}

Define an operator function $T: \Omega \rightarrow \mathcal{L}(H_p^1(D_0),H_p^1(D_0))$  by
\begin{equation}\label{TEoperator}
T(\omega):=I -B(\omega),
\end{equation}
where $I$ is the identity operator.
Since \eqref{operator form} holds for all $v\in H^1_p(D_0)$, \eqref{TE weak form} is equivalent to finding $(\omega, u) \in \Omega \times H_p^1(D_0)$ such that
\begin{equation}\label{TEoperatorEig}
T(\omega)u =u-B(\omega)u=0.
\end{equation}
Since $\epsilon(x, \omega)$ is holomorphic in $\omega$,
$T(\omega)$ is a holomorphic operator function \cite{gohberg2009, engstrom2010}.



The following lemma claims that  $B(\omega)$ is a compact operator.
It is proved in \cite{engstrom2010} (Lemma 4.2 therein, see also Theorem 3.3.3 in \cite{SunZhou2016} for a similar result), which uses the Cauchy sequence and the boundedness 
of the sesquilinear form $b$. 
\begin{lemma}
The operator $B(\omega)$ is compact for all $\omega\in \Omega$.
\end{lemma}
\section{FEM Discretization and Convergence}
In this section, we employ Lagrange elements to discretize the operators and prove
the convergence of eigenvalues using the abstract approximation theory by Karma \cite{karma1996I, karma1996II}.
For simplicity, we shall focus on the linear Lagrange element and the results for higher order elements are similar.

Let $\mathcal{T}_h$ be a regular triangular mesh for $D_0$, where $h$ is the mesh size.
Let $V_h\subset H_p^1(D_0)$ be the assocoated linear Lagrange finite element space. 
Let $B^h: \Omega \rightarrow\mathcal{L}(V_h,V_h)$ be given by
\[
(B^h(\omega) u_h,v_h)_1:=b(\omega; u_h,v_h) \quad \text{for all } v_h\in V_h.
\]
Define the discrete operator function $T^h: \Omega \rightarrow \mathcal{L}(V_h,V_h)$ such that
\[
T^h(\omega)=I^h-B^h(\omega),
\]
where $I^h$ is the identity operator from $V_h$ to $V_h$.

Define the linear projection operator
$p_h: H_p^1(D_0)\rightarrow V_h$ such that
\begin{equation}\label{projection}
(u,v_h)_1=(p_h u,v_h)_1~~~\text{for all } v_h\in V_h.
\end{equation}
Clearly, $p_h$ is bounded and $\|p_h u-u\|_1\rightarrow 0$ as $h\to 0$ for $u\in H_p^1(D_0)$ (see, e.g., Sec. 3.2 of \cite{SunZhou2016}).
Furthermore, using the Aubin-Nitsche Lemma (Theorem 3.2.4 of \cite{SunZhou2016}), it holds that
\begin{equation}\label{L2 error}
\|u- p_h u\|_{L^2(D_0)}\leq Ch \|u\|_{H^1(D_0)}.
\end{equation}

\begin{lemma}\label{supsupTh}
There exists $h_0>0$ small enough such that, for every compact set $\Omega \subset \mathbb{C}$,
\begin{equation}\label{bounded}
\sup\limits_{h<h_0}\sup\limits_{\omega\in \Omega}\|T^h(\omega)\|_{\mathcal{L}(V_h,V_h)}<\infty.
\end{equation}
\end{lemma}
\begin{proof}
Due to the boundedness of $B^h(\omega)$ and compactness of $\Omega$, it holds that
\[
\|T^h(\omega)v_h\|_1=\|v_h-B^h(\omega)v_h\|_1\leq \|v_h\|_1+\|B^h(\omega)v_h\|_1\leq C\|v_h\|_1, \quad v_h \in V_h.
\]
\end{proof}

\begin{lemma}\label{ThphphT}
Let $g\in H^{1}_p(D_0)$ and $\omega\in \Omega$.
Then 
\begin{equation}\label{ThTE}
\|T^h(\omega)p_h g-p_hT(\omega)g\|_{H^1(D_0)} \leq Ch \|g\|_{H^{1}(D_0)}.
\end{equation}
\end{lemma}
\begin{proof}  
By the approximation property of $p_h$ in \eqref{L2 error} and \eqref{boundness}, letting $v_h \in V_h$, we have that
\begin{align*} 
&\left(T^h(\omega)p_h g -p_h T(\omega)g, v_h\right)_1\\
=&(p_h g,v_h)_1-(B^h p_h g,v_h)_1-(p_h g,v_h)_1+(p_h B(\omega)g,v_h)_1\\
=&(B(\omega)g,v_h)_1-(B^h p_h g,v_h)_1\\
=&b(\omega; g,v_h)-b(\omega; p_h g,v_h)\\
=&\left( \left(\frac{\omega^2}{c^2}\epsilon+1\right)(g-p_h g), v_h\right)-\left(2i(g-p_h g){\boldsymbol k}, \nabla v_h\right)-\left(|{\boldsymbol k}|^2(g-p_h g), v_h \right)\\
\leq& C\|g-p_h g\|_{L^2(D_0)}\|v_h\|_{H^1(D_0)}\\
\leq& Ch \|g\|_{H^{1}(D_0)}\|v_h\|_{H^1(D_0)},
\end{align*}
which implies \eqref{ThTE}.
\end{proof}\\
\textbf{Remark.} If $g\in H^1_p(D_0)\cap H^{1+\sigma}(D_0)$ for some $0<\sigma \le 1$, 
then 
\[
\|T^h(\omega)p_h g-p_h T(\omega)g\|_{H^1(D_0)} \leq Ch^{1+\sigma}\|g\|_{H^{1+\sigma}(D_0)}.
\]

Now we are ready to present the main convergence theorem for the eigenvalues.
\begin{theorem}
Let $\omega\in \sigma(T)$ and $h$ be small enough.
Then there exists a sequence $\{\omega_h\in\sigma(T^h)\}$
such that $\omega_h\rightarrow\omega$ as $h\rightarrow 0$ and
\begin{equation}\label{err1}
 |\omega_h-\omega|\leq Ch^{\frac{1}{r_0}}.
\end{equation}
Furthermore, if $G(\omega) \subset H^1_p(D_0)\cap H^{1+\sigma}(D_0)$, $0<\sigma \le 1$, then 
\begin{equation}\label{err2}
 |\omega_h-\omega|\leq Ch^{\frac{1+\sigma}{r_0}}.
\end{equation}
\end{theorem}
\begin{proof}
Let $\{h_n\}$ be a small enough monotonically decreasing sequence of positive numbers such that
$h_n\rightarrow 0$ as $n\rightarrow\infty$. 
Correspondingly, we define sequences of operators $T_n(\omega):=T^{h_n}(\omega)$, 
finite element spaces $V_n:=V_{h_n}$, and the projections $p_n:=p_{h_n}$.
Clearly, (A1) in Section 2 holds with $X = Y = H_p^1(D_0)$, $X_n = Y_n = V_n$, and $q_n = p_n$. 
Conditions (A2) and (A3) hold due to Lemma 4.1 and Lemma 4.2.

Next, we verify (A4). Without loss of generality, assume that $v_n \in V_n$, $n\in \mathbb{N}$  with $\Vert v_n\Vert_1 \le 1$ and
\begin{equation}\label{assumption3}
\lim\limits_{n\rightarrow\infty}\|T_n(\omega)v_n-p_ny\|_1=0,
\end{equation}
for some $y\in H_p^1(D_0)$. For any subsequence $\{x_n\}, n \in N \subset \mathbb N$, we shall show that there exists a subsequence $N' \subset N$
and a $v\in H_p^1(D_0)$ such that
\begin{equation}\label{conclusion}
\lim\limits_{N' \ni n\to \infty} \|v_n - p_n v\|_1=0
\end{equation}
by considering $\omega\in \rho(T)$ and $\omega\in \sigma(T)$ separately.

If $\omega\in \rho(T)$, then $T(\omega)^{-1}$ exists and is bounded. Let $v = T(\omega)^{-1}y$.
It holds that
\begin{align}\label{eqn3}
&v_n - p_n v \nonumber \\
=& T(\omega)^{-1}\big((T(\omega)-T_n(\omega))(v_n - p_nv) + T_n(\omega)v_n - p_n T(\omega) v + p_n T(\omega) v- T_n(\omega) p_n v\big)\nonumber\\
=&T(\omega)^{-1}\big((B_n(\omega)-B(\omega))(v_n - p_nv) + T_n(\omega)v_n - p_n T(\omega) v + p_n T(\omega) v- T_n(\omega) p_n v\big).
\end{align}
By the definitions of $B(\omega)$ and $B_n(\omega)$, for $g_n \in V_n$, we have the following Galerkin orthogonality
\begin{align*}
\big((B_n(\omega) - B(\omega))g_n, v_n\big)_1 = 0 \quad \text{for all } v_n \in V_n.
\end{align*}
Consequently,
\begin{align*}
\big((B_n(\omega) - B(\omega))g_n, (B_n(\omega) - B(\omega))g_n\big)_1 
=& \big((B_n(\omega) - B(\omega))g_n, -B(\omega)g_n + p_n B(\omega)g_n\big)_1 \\
\leqslant& \Vert (B_n(\omega) - B(\omega))g_n\Vert_1 \Vert (I - p_n)B(\omega)g_n\Vert_1,
\end{align*}
which implies that $\Vert (B_n(\omega) - B(\omega))|_{V_n}\Vert \leqslant \Vert (I - p_n)B(\omega)|_{V_n}\Vert \leqslant \Vert (I-p_n) B(\omega)\Vert$.
The last term goes to zero
since $I - p_n$ converges to zero pointwisely and $B(\omega)$ is compact 
(both are viewed as operators from $H^1_p (D_0)$ to itself). Hence $\Vert (B_n(\omega) - B(\omega))|_{V_n}\Vert\rightarrow 0$ 
and, for $n$ large enough, $\Vert T(\omega)^{-1}(B_n(\omega) - B(\omega))|_{V_n}\Vert \leqslant 1/2$.
Using \eqref{assumption3} and Lemma 4.2, one has that
\begin{align*}
\Vert v_n - p_n v\Vert_1 \leqslant C\Vert T_n(\omega) v_n - p_n T(\omega) v\Vert_{1} + C \Vert p_n T(\omega) v - T_n(\omega) p_n v\Vert_{1} \rightarrow 0.
\end{align*}

If $\omega\in \sigma(T)$, let $\mathcal{N}(\omega):=\mathcal{N}(T(\omega))$ be the finite dimensional eigenspace associated with $\omega$ \cite{karma1996I}.
We denote by $P_{\mathcal{N}(\omega)}$ the projection from $H_p^1(D_0)$ to $\mathcal{N}(\omega)$, by $T(\omega)^{-1}$ the inverse of $T(\omega)|_{H_p^1(D_0)/\mathcal{N}(\omega)}$
from $\mathcal{R}(T(\omega))$ to ${H_p^1(D_0)/\mathcal{N}(\omega)}$. 
Similarly, using \eqref{assumption3} and the approximation property of $p_n$, we have that
\begin{align*}
\Vert T(\omega) v_n - y\Vert_1 &\leq \Vert T(\omega)v_n - T_n(\omega)v_n\Vert_1 + \Vert T_n(\omega)v_n - p_n y\Vert_1 + \Vert p_n y-y\Vert_1\nonumber\\
&=\Vert (B_n(\omega)-B(\omega))v_n\Vert_1+\Vert T_n(\omega)v_n - p_n y\Vert_1 + \Vert p_n y-y\Vert_1\rightarrow 0, \quad n \to \infty.
\end{align*}
Because $T(\omega)$ is a Fredholm operator, $\mathcal{R}(T(\omega))$ is closed and thus $y \in \mathcal{R}(T(\omega))$.

Let  $v^{\prime} := T(\omega)^{-1} y$ and $v_n^{\prime}:= (I-p_n P_{\mathcal{N}(\omega)})v_n$. 
Similar to \eqref{eqn3}, as $n\rightarrow \infty$ , we deduce that
\begin{equation*}
v_n^{\prime} - p_n v^{\prime} = T(\omega)^{-1}\big((T(\omega)-T_n(\omega))(v_n^{\prime} - p_nv^{\prime}) + T_n(\omega)v_n^{\prime} - p_n T(\omega) v^{\prime} + p_n T(\omega) v^{\prime}- T_n(\omega) p_n v^{\prime}\big) \rightarrow 0.
\end{equation*}
On the other hand, since $\mathcal{N}(\omega)$ is finite dimensional, there is a subsequence $N' \subset N$ and 
$v^{\prime\prime}\in \mathcal{N}(\omega)$ such that $\Vert P_{\mathcal{N}(\omega)}v_n - v^{\prime\prime}\Vert_1\rightarrow 0$ as $N' \ni n \rightarrow \infty$. 
Therefore, letting $v := v^{\prime} + v^{\prime\prime}$, we have that
\begin{align*}
\Vert v_n - p_n v\Vert_1 \leq \Vert v_n^{\prime} - p_n v^{\prime}\Vert_1 + \Vert p_n P_{\mathcal{N}(\omega)}v_n - p_n v^{\prime\prime}\Vert_1 \rightarrow 0,\quad \text{as}\quad N' \ni n\rightarrow \infty
\end{align*}
and (A4) is verified. 

The quantity $\varepsilon_n$ is the consistency error defined by
\begin{equation}
\varepsilon_n=\max\limits_{|\eta-\omega|\leq \delta} \max\limits_{g\in G(\omega)}\|T_n(\eta) p_n g-q_n T(\eta)g \|_{H^1(D_0)},
\end{equation}
where $\delta>0$ is sufficiently small. 
From the proof of Lemma~\ref{ThphphT}, one has that
$$\varepsilon_n \leq Ch,$$
and \eqref{err1} follows Theorem \ref{theorem1} directly.
Moreover, if $G(\omega) \subset H^1_p(D_0)\cap H^{1+\sigma}(D_0)$, then
$$\varepsilon_n \leq Ch^{1+\sigma},$$
and \eqref{err2} follows.
\end{proof}

\begin{corollary}  For a simple eigenvalue $\omega$ ($r_0=1$), according to \eqref{err1} and \eqref{err2}, 
\begin{equation}\label{simple con order}
|\omega-\omega_h|\leq Ch \quad {\rm or} \quad |\omega-\omega_h|\leq Ch^{1+\sigma},
\end{equation}
respectively.
\end{corollary}


\section{Spectral Indicator Method}
To obtain the band gap structures, one needs an effective scheme to approximate the eigenvalues of a holomorphic operator function.
To this end, we develop a spectral indicator method (SIM) to compute the eigenvalues of $T^h$ in a bounded simply connected region $\Omega \subset\mathbb{C}$.
We refer the readers to \cite{gong2020, xiao2020} for the applications of SIM to the Dirichlet eigenvalue problem and transmission eigenvalue problem.

Let $\phi_i, i=1,\cdots,N$, be the basis functions for $V_h$.
For a fixed ${\boldsymbol k} \in\mathcal{K}$, let $M^h(\omega),  M^h, A^h, S^h$ be the matrices corresponding
to the terms 
\[
\int_{D_0} \epsilon(\omega) u_h \overline{v}_h dx, ~\int_{D_0} u_h \overline{v}_h dx,\quad \int_{D_0} \nabla u_h \cdot \nabla \overline{v}_h dx, \quad\int_{D_0}  u_h {\boldsymbol k} \cdot \nabla \overline{v}_h dx.
\]
In matrix form, $T^h(\omega)$ is given by
\begin{equation}\label{TE nonlinear eig} 
{\mathbb T}^h(\omega):=A^h+2iS^h+|{\boldsymbol k}|^2 M^h-\left(\frac{\omega}{c}\right)^2 M^h(\omega).
\end{equation}


Without loss of generality, let $\Omega\subset\mathbb{C}$ be a square and $\Theta$ be the circle circumscribing 
$\Omega$ lying in $\rho(\mathbb{T}^h)$. Thus ${\mathbb T}^h(\omega)^{-1}$ exists and is bounded for all $\omega \in\Theta$. 
Define an operator $\Pi$ by
\begin{equation}
\Pi=\frac{1}{2\pi i}\int_\Theta {\mathbb T}^h(\omega)^{-1} d\omega.
\end{equation}
Let $g_h = \sum_{i=1}^N g_i \phi_i \in V_h$ be random and $\vec{g}_h = (g_1, \ldots, g_N)^T$. If ${\mathbb T}^h$ has no eigenvalue in $\Omega$, $\Pi \vec{g}_h=0$. 
On the other hand, if ${\mathbb T}^h$ has at least one eigenvalue in $\Omega$, $\Pi \vec{g}_h\neq 0$ almost surely. 
This fact is used to decide if $\Omega$ contains eigenvalues, which is the key idea of SIM \cite{huang2016, huang2018, Huang2020}.

Let $\vec{x}_h$ be the solution of ${\mathbb T}^h(\omega) \vec{x}_h(\omega)=\vec{g}_h$. Using the trapezoidal rule to 
approximate $\Pi \vec{g}_h$, we define an indicator for $\Omega$ by
\[
I_\Omega := \left|\Pi \vec{g}_h \right| \approx \bigg|\frac{1}{2\pi i}\sum\limits_{j=1}^{m_0}w_j \vec{x}_h(\omega_j)\bigg|,
\]
where $m_0$ is the number of quadrature points and $w_j$ are the weights. 
One computes the indicator $I_\Omega$ and decides if $\Omega$ contains eigenvalue(s) or not. 
In implementation, let $\delta_0>0$ be the indicator threshold. For a square $\Omega^i$ at level $i$ ($\Omega^1 = \Omega$), if $I_{\Omega_i} \le \delta_0$,
$\Omega_i$ is discarded.
If $I_{\Omega_i}>\delta_0$, one can uniformly divide $\Omega^i$ into (four) small squares $\Omega^i_j, j=1,2,3,4,$
and then compute the indicators $ I_{\Omega^i_j}$ using the circumscribing circle $\Theta^i_j$ of $\Omega^i_j$.
The procedure continues until the size of $\Omega^i_j$ are less than the precision $\beta_0$. Consequently, eigenvalues are identified with precision $\beta_0$.

The following algorithm SIM-H (spectral indicator method for holomorphic functions) computes all the eigenvalues of $T^h$ in $\Omega$
(see also \cite{gong2020}).

\begin{itemize}
\item[] {\bf SIM-H:}
\item[-] Given a domain $D_0$, a square $\Omega \subset \mathbb C$, and ${\boldsymbol k} \in \mathcal{K}$.
\item[-] Choose  the precision $\beta_0$, the indicator threshold $\delta_0$.
\item[1.] Generate a triangular mesh for $D_0$ and the matrices $M^h, A^h, S^h$.
\item[2.] Choose a random $\vec{g}_h$.
\item[3.] While $d(\Omega^i) > \beta_0$, where $d(\Omega^i)$ is the diameter of $\Omega^i$'s, do
	\begin{itemize}
	\item For each square $\Omega^i_k$ at current level $i$, evaluate the indicator $I_{\Omega^i_k}$ as follows.
		\begin{itemize}
			\item At each quadrature point $\omega_j$, assemble $M^h(\omega_j)$.
			\item Solve $\vec{x}_h(\omega_j)$ for ${\mathbb T}^h(\omega_j)  \vec{x}_h(\omega_j) = \vec{g}_h$.
			\item Compute
		\[
			I_{\Omega^i_k} :=  \bigg|\frac{1}{2\pi i}\sum\limits_{j=1}^{m_0}w_j \vec{x}_h(\omega_j)\bigg|.
		\]
		\end{itemize}
	\item If $ I_{\Omega^i_k} > \delta_0$, uniformly divide $\Omega^i_k$ into smaller squares. Otherwise, discard $\Omega^i_k$.
	\item Collect all the squares and leave them to the next level $i+1$.
	\end{itemize}
\item[4.] Output the eigenvalues (centers of the small squares).
\end{itemize}

In all numerical examples, we set $\delta_0=0.01$ and $\beta_0=10^{-4}$. We note that $\delta_0$ is usually problem dependent and can be chosen by test and error.


\section{Numerical Examples}
In this section, we present several examples by showing the dispersion relations $\omega({\boldsymbol k})$ with ${\boldsymbol k}$ 
moving along $\overline{M_1M_3}$,  $\overline{M_3M_5}$ and $\overline{M_5M_1}$ (Fig.~\ref{domain and k}: right)
and the convergence orders of the eigenvalues of \eqref{TE weak form}.
Consider an infinite array of identical, infinitely long, parallel cylinders, embedded in vacuum, 
whose intersection with a perpendicular plane forms a simple square lattice with a disc inside (Fig.~\ref{domain and k}: left).
In all the examples, $\epsilon_a=1$ is the permittivity of the vacuum and
$\epsilon_b$ is the permittivity in the disc.
Define the filling fraction $f=\pi r^2/a^2$, where $a=1$ is the side length of the unit cell $D_0$, $r$ is the radius of the disc at the center of $D_0$. 
The relative error of the computed eigenvalues on a series of uniformly refined meshes $\mathcal{T}_{h_i}$ is defined as
\[
\xi_{i+1}=\frac{|\omega_{h_{i}}-\omega_{h_{i+1}}|}{\omega_{h_{i+1}}},
\]
where $h_{i+1}=\frac{1}{2}h_i$. 
Note that $V_h\subset H_p^1(D_0)$ is the linear Lagrange element space.

$\textbf{Example 1}$. 
We first compute an example from \cite{dobson1999} for validation, where $r$=0.378 ($f\approx$ 0.448). The permittivity $\epsilon_b=8.9$ is independent of $\omega$. 
Thus the operator $T(\omega)$ is linear.
We identify $\mathbb C$ with $\mathbb R^2$ and choose $\Omega=[0.2, 9.8] \times [-4.8, 4.8]$. 
Fig.~\ref{example1 result} shows the band structure computed on a mesh $\mathcal{T}_h$ with $h\approx\frac{1}{20}$,
which reveals the existence of band gaps. The figure is the same as Fig.~2. of \cite{dobson1999} .

\begin{figure}[h!]
\begin{center}
\begin{tabular}{c}
\resizebox{0.5\textwidth}{!}{\includegraphics{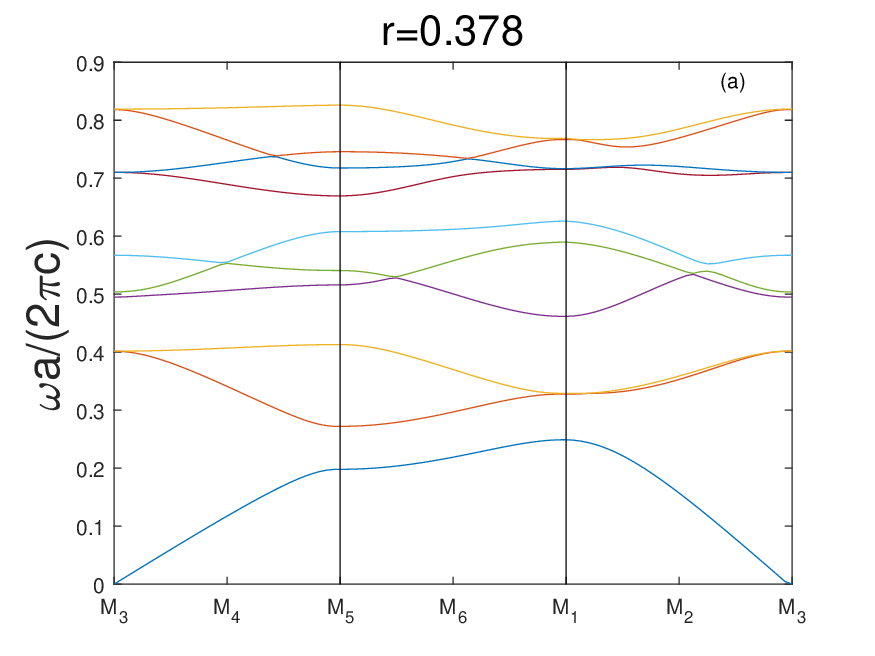}}
\end{tabular}
\end{center}
\caption{Band structure: $\epsilon_b= 8.9, \epsilon_a = 1$, $r=0.378$ $(f\approx 0.448)$.}
\label{example1 result}
\end{figure} 


In Table \ref{E1CO}, we show the relative error and convergence order of the first eigenvalue for ${\boldsymbol k}=(\pi,\pi)$
in \eqref{TE weak form}
which corresponds to $M_1$ in Fig.~\ref{example1 result}.
Second order convergence is obtained.
\begin{table}
\centering
\begin{tabular}{cccc}
\hline
$h$  &$\omega/ (2\pi c)$  &$\xi$  &$order$  \\
\hline\hline
1/10    &0.2539          &-           &-                     \\
1/20    &0.2490           &0.0197    &-                 \\   
1/40    &0.2477          &0.0055   &1.8359        \\
1/80    &0.2473          &0.0014   &1.9496     \\
\hline
\end{tabular}
\caption{Example~1: The first eigenvalue, relative error, and convergence order (${\boldsymbol k}=(\pi,\pi)$).}
\label{E1CO}
\end{table}

$\textbf{Example 2}$. We consider an example from \cite{kuzmiak1997} for a frequency dependent material. 
Let 
\begin{equation}\label{Exepsb}
\epsilon_b(\omega)=\epsilon_\infty\frac{\omega_L^2-\omega^2}{\omega_T^2-\omega^2},
\end{equation}
where $\epsilon_\infty$ is the optical frequency dielectric constant, $\omega_L$ and $\omega_T$ 
are the frequencies of the longitudinal and transverse optical vibration modes of infinite wavelength, respectively.
In \eqref{Exepsb}, $\epsilon_\infty=10.9$, $\omega_T=8.12$THz, $\omega_L=8.75$THz, and $\frac{\omega_Ta}{2\pi c}=1$.
We choose $f=0.001 ~(r=0.0178)$ and $f=0.1 ~(r=0.1784)$. 

\begin{figure}[h!]
\begin{center}
\begin{tabular}{lll}
\resizebox{0.45\textwidth}{!}{\includegraphics{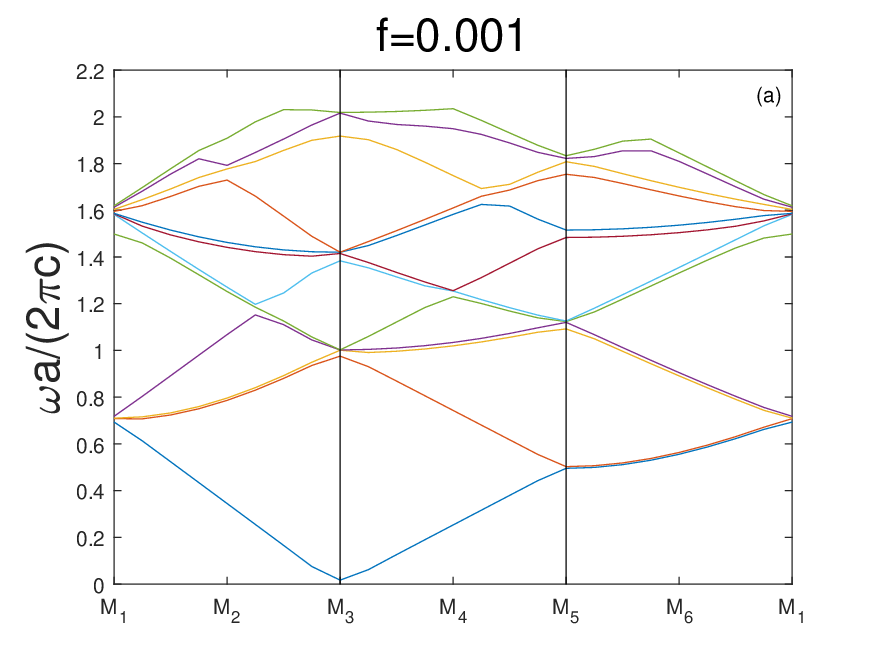}}&
\resizebox{0.45\textwidth}{!}{\includegraphics{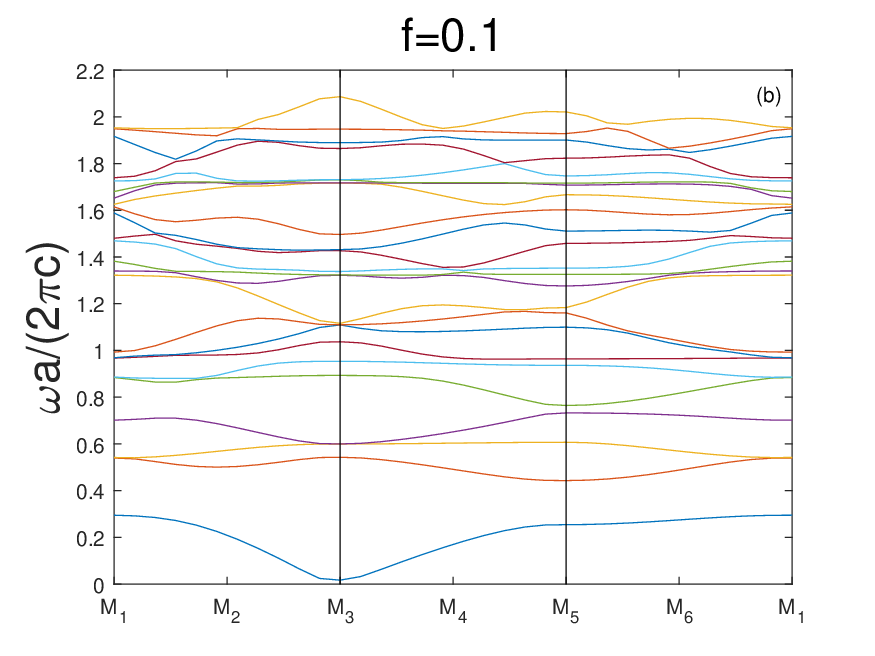}}
\end{tabular}
\end{center}
\caption{Band structures of square lattice of GaAs cylinders in vacuum.
Left: $f=0.001$. Right: $f=0.1$. The result is consistent with Fig.~1 of \cite{kuzmiak1997}}
\label{GaAs TE}
\end{figure} 



Let $\Omega=[0.1, 13.9] \times[-6.9, 6.9]$.
The band structures shown in Fig.~\ref{GaAs TE}, which is consistent with Fig.~1 in \cite{kuzmiak1997}.
The first eigenvalues for ${\boldsymbol k}=(\pi, \pi)$ and $r=0.1784$ are listed in Table \ref{example2 con order} and 
second order convergence is achieved.
\begin{table}
\centering
\begin{tabular}{cccc}
\hline
$h$  &$\omega/ (2\pi c)$  &$\xi$  &order \\
\hline\hline
1/10    &0.3038          &-           &-                     \\
1/20    &0.2949           &0.0301    &-                 \\   
1/40    &0.2925          &0.0080   &1.9074         \\
1/80    &0.2919          &0.0021   &1.9679     \\
\hline
\end{tabular}
\caption{Example~2: The first eigenvalue, relative error, and convergence order (${\boldsymbol k}=(\pi,\pi)$).}
\label{example2 con order}
\end{table}

$\textbf{Example 3}$.  We consider both lossless and lossy metallic components characterized by a complex frequency dependent dielectric function
\begin{equation}\label{parameter2}
\epsilon_b(\omega)=1-\frac{\omega_p^2}{\omega(\omega+i\gamma)},
\end{equation}
where $\omega_p$ is the plasma frequency, $\gamma$ is an inverse electronic relaxation time \cite{kuzmiak1994}. 

{\it The lossless case}  ($\gamma=0$). 
This case corresponds to a frequency-dependent real-valued dielectric function.
Taking $\frac{\omega_pa}{2\pi c}=1$ and $\Omega=[0.2, 11.8]\times[-5.8, 5.8]$, 
we show the photonic band structures with $f=0.001~(r=0.0178)$ and $f=0.7~(r=0.4720)$ in Fig.~\ref{TE lossless result}. 
When the filling fraction is small, the band structure 
is not significantly different from the dispersion curves for electromagnetic waves in vacuum.
However, for higher value of filling fraction, the band structure differs substantially and reveals the existence of band gaps. 
\begin{figure}[h!]
\begin{center}
\begin{tabular}{lll}
\resizebox{0.45\textwidth}{!}{\includegraphics{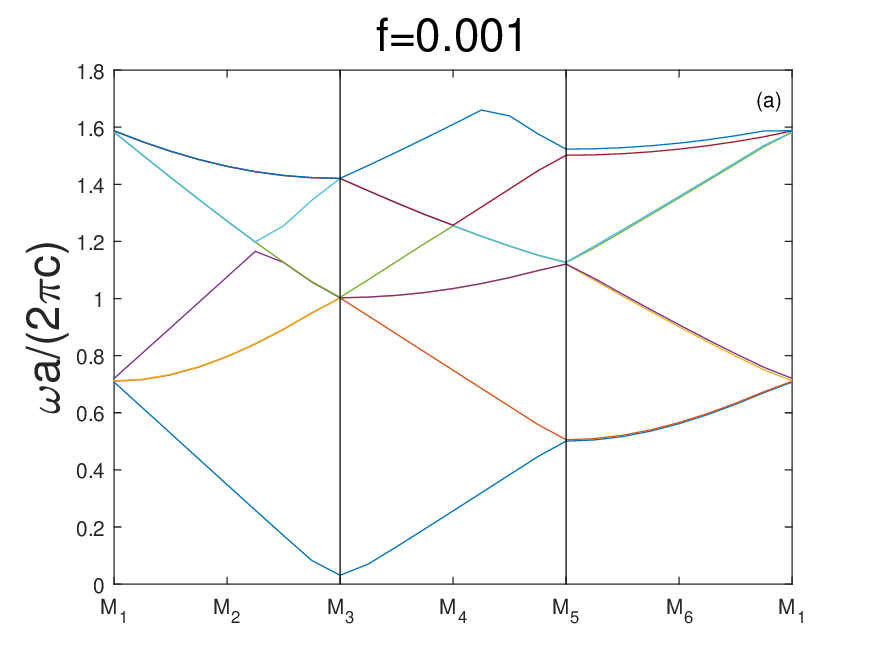}}&
\resizebox{0.45\textwidth}{!}{\includegraphics{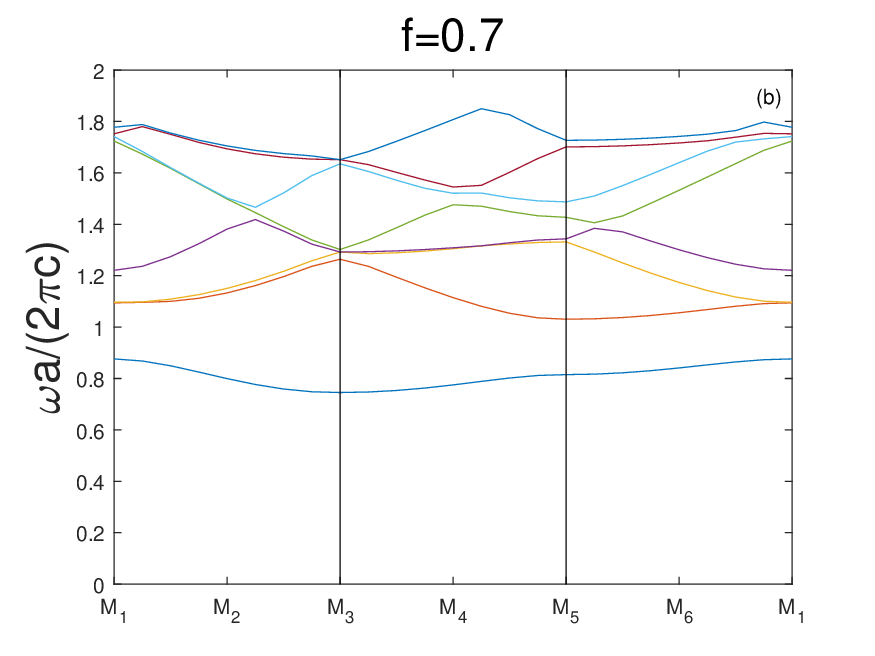}}
\end{tabular}
\end{center}
\caption{Band structures for the square lattice of lossless metal cylinders in vacuum. Left: $f=0.001$. Right: $f=0.7$. The result is consistent with Fig.~1 in \cite{kuzmiak1994}}
\label{TE lossless result}
\end{figure}

This problem was computed previously by the plane wave expansion method in \cite{kuzmiak1994}.
Fig.~\ref{TE lossless result} is consistent with  Fig.~1 of \cite{kuzmiak1994}


{\it The lossy case} ($\gamma\neq0$). This case yields complex eigenvalues. 
We take the real parts of eigenvalues to produce the band structures as \cite{kuzmiak1998}. Let $\gamma=0.01\omega_p$. 
Fig.~\ref{lossy result} shows the band structures when $f=0.01~(r=0.0564)$ and $f=0.1~(r=0.1784)$, 
which is consistent with  Fig.~6 (a) and (b) of \cite{kuzmiak1998}.
Again, we obtain second order convergence for both lossless and lossy media in Table~\ref{E3lossless} and Table~\ref{E3lossy}, respectively.


\begin{figure}[h!]
\begin{center}
\begin{tabular}{lll}
\resizebox{0.45\textwidth}{!}{\includegraphics{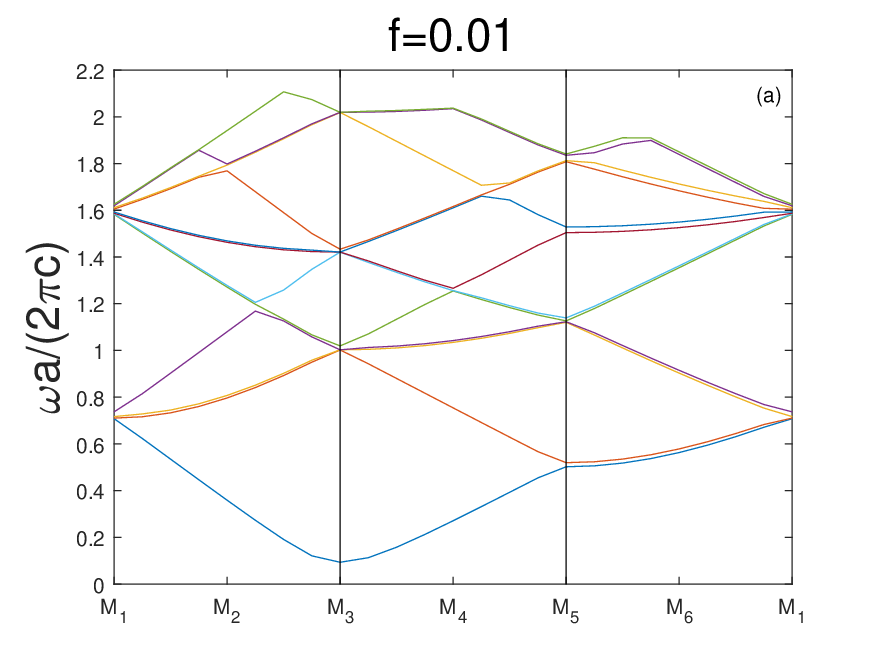}}&
\resizebox{0.45\textwidth}{!}{\includegraphics{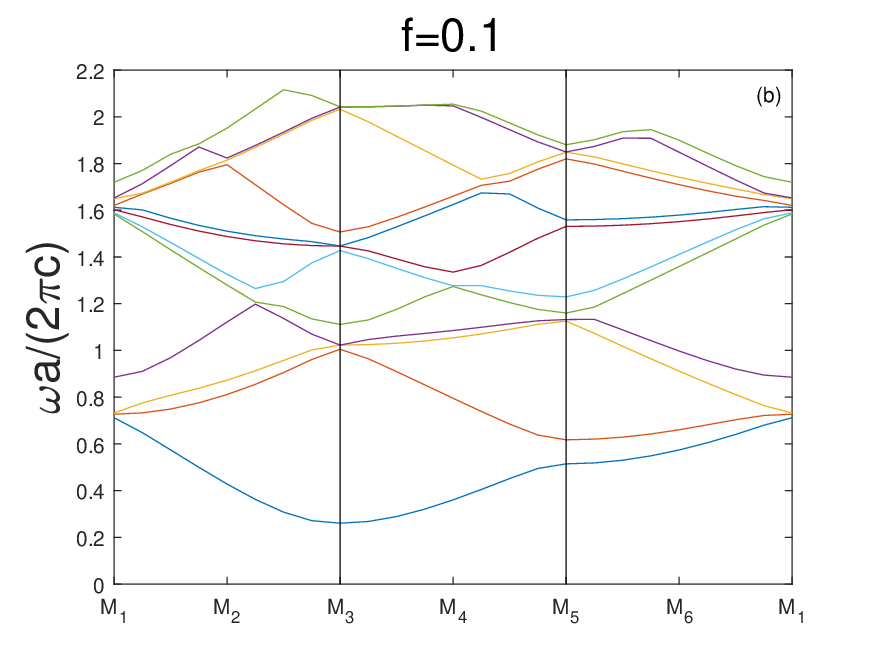}}
\end{tabular}
\end{center}
\caption{Band structures for the square lattice of lossy metal cylinders in vacuum. 
Left: $f=0.01$. Right: $f=0.1$.  The result is the same as Fig.~6 (a) and (b) in \cite{kuzmiak1998}.}
\label{lossy result}
\end{figure}

\begin{table}[h]
\centering
\begin{tabular}{cccc}
\hline
$h$  &$\omega/(2\pi c)$  &$\xi$  & order  \\
\hline\hline
1/10    &0.8878         &-           &-                     \\
1/20    &0.8762           &0.0133   &-                 \\   
1/40    &0.8730          &0.0036   &1.8745         \\
1/80    &0.8722          &9.3047e-04   &1.9589     \\
\hline
\end{tabular}
\caption{Example 3 (lossless case): The first eigenvalues, relative errors, and convergence orders  (${\boldsymbol k}=(\pi,\pi)$, $f=0.7$).}
\label{E3lossless}
\end{table}
\begin{table}[h]
\centering
\begin{tabular}{cccc}
\hline
$h$  &$\omega/c$  &$\xi$  &order  \\
\hline\hline
1/10    &1.6322-0.0220i         &-           &-                     \\
1/20    &1.6384-0.0217i           &0.0038    &-                 \\   
1/40    &1.6398-0.0216i          &8.8922e-04   &2.0984         \\
1/80    &1.6402-0.0216i          &2.1601e-04   &2.0415     \\
\hline
\end{tabular}
\caption{Example 3 (lossy case): The first eigenvalues, relative errors, and convergence orders (${\boldsymbol k}=(0,0)$, $f=0.1$).}
\label{E3lossy}
\end{table}

\section{Conclusion}\label{conclusion}  
We propose a new finite element approach to compute photonic band structures for dispersive materials. 
The problem is first transformed to the eigenvalue problem of a holomorphic Fredholm operator function of index zero.
Lagrange finite elements are employed to discretize the operator functions. Using the abstract approximation theory for 
holomorphic operator functions \cite{karma1996I, karma1996II}, we  prove the convergence of eigenvalues.
Numerical results validate the theory and are consistent with those in literature.

Solving nonlinear eigenvalue problems is a challenging research topic, even for the finite dimensional cases \cite{mehrmann2004}. 
The standard linearization methods \cite{mackey2006} enlarge the matrix dimension, 
thus increasing the computation and storage requirements. 
Methods based on Newton's iteration often need a good initial guess.
However, in general, little is known about the spectral distribution for the nonlinear eigenvalue problem.
SIM is a practical way to compute eigenvalues of (nonlinear) holomorphic operator functions 
and requires no a priori information of the spectrum.

In this paper, we only consider the TE case. The TM case is more delicate and is currently under our investigation.
Another interesting topic is to extend the proposed method to the 3D case.

\section*{Acknowledgement}
The research of B. Gong is supported partially by China Postdoctoral Science Foundation Grant 2019M650460.
The research of Z. Zhang is supported partially by the National Natural Science Foundation of China grants NSFC 11871092, NSAF U1930402 and NSFC 11926356.

\end{document}